\documentclass[12pt]{amsart}

\usepackage[margin=1in]{geometry}
\usepackage{setspace}
\usepackage{amsmath}
\usepackage{amsxtra}
\usepackage{amscd}
\usepackage{amsthm}
\usepackage{amssymb}
\usepackage[normalem]{ulem}
\usepackage{comment}
\usepackage{color}
\usepackage{tikz}
\usetikzlibrary{matrix,arrows,decorations.pathmorphing,automata, matrix, positioning, calc, shapes.multipart}
\usepackage{graphicx}
\usepackage{hyperref}
\usepackage[shortlabels]{enumitem}
\usepackage{float}
\usepackage[vcentermath]{youngtab}

\numberwithin{equation}{section}

\theoremstyle{definition}
\newtheorem{theorem}[equation]{Theorem}
\newtheorem{lemma}[equation]{Lemma}
\newtheorem{proposition}[equation]{Proposition}

\newtheorem{remark}[equation]{Remark}

\newcommand{\sle}{\widehat{\mathfrak{sl}_e}}

\newcommand{\slinf}{\mathfrak{sl}_\infty}

\newcommand{\ct}{\mathrm{ct}}

\newcommand{\ul}{\underline}
\newcommand{\cA}{\mathcal{A}}
\newcommand{\cF}{\mathcal{F}}
\newcommand{\cO}{\mathcal{O}}
\newcommand{\N}{\mathbb{N}}

\usepackage[foot]{amsaddr}

\title[Finite-dimensional unitary representations of type $B$ RCA]{Classification of the finite-dimensional unitary representations of type $B$ rational Cherednik algebras}
\author{Emily Norton}
\address[E.N.]{Max Planck Institute for Mathematics, Vivatsgasse 7, 53111 Bonn, Germany.}
\email{enorton@mpim-bonn.mpg.de}

\begin{document}

\begin{abstract}
We compare crystal combinatorics of the level $2$ Fock space with the classification of unitary representations of type $B$ rational Cherednik algebras to show that any finite-dimensional unitary irreducible representation of such an algebra is labeled by a bipartition consisting of a rectangular partition in one component and the empty partition in the other component. This is a new proof of a result that can be deduced from theorems of Montarani and Etingof-Stoica.
\end{abstract}

\maketitle

\section{Statement of the result}
In this note we answer the following question about representations in the category $\cO$ of rational Cherednik algebras associated to the Weyl group $B_n$: which unitary representations of these algebras are also finite-dimensional? To give a precise combinatorial answer in terms of the labeling of irreducible representations by bipartitions of $n$, we use  a criterion by Etingof and Stoica in terms of the $c$-function \cite{EtingofStoica}, the classification of unitary representations by Griffeth \cite{Griffeth}, and the combinatorial description of finite-dimensional representations in terms of source vertices of crystals on a level $2$ Fock space due to Shan, Vasserot, Losev, Gerber, and the author \cite{Shan},\cite{ShanVasserot},\cite{Losev},\cite{Gerber1},\cite{Gerber2},\cite{GerberN}. Let $(e,\ul{s})$ be a parameter for a level $2$ Fock space, so $e\in\mathbb{Z}_{\geq 2}$ is the rank and $\ul{s}=(s_1,s_2)\in\mathbb{Z}^2$ is the charge. To this datum we can associate a rational Cherednik algebra $H_{e,\ul{s}}(B_n)$ and its category $\cO$ of representations, denoted $\cO_{e,\ul{s}}(n)$.
\begin{theorem}\label{unitary+fd} Let $\ul{\lambda}$ be a bipartition. \begin{enumerate}
\item There exists a parameter $(e,\ul{s})\in\mathbb{Z}_{\geq 2}\times \mathbb{Z}^2$ such that the irreducible representation $L(\ul{\lambda})\in\cO_{e,\ul{s}}(n)$ is both unitary and finite-dimensional
 if and only if $\ul{\lambda}=(\lambda,\emptyset)$ or $\ul{\lambda}=(\emptyset,\lambda)$ and $\lambda$ is a rectangle. 
 \item For a fixed parameter $(e,\ul{s})$ and a rectangle $\lambda$, the irreducible representation $L(\lambda,\emptyset)$ is both unitary and finite-dimensional if and only if $r-q=s-e$, where $r$ is the number of rows and $q$ the number of columns of $\lambda$, and $s=s_2-s_1$.
 \end{enumerate}
 A similar statement to (2) for $L(\emptyset,\lambda)$ holds, see Remark \ref{switch}.
\end{theorem}

It turns out that Theorem \ref{unitary+fd} can be deduced from a result of Montarani on wreath product algebras \cite{Montarani} together with the afore-mentioned criterion of Etingof-Stoica \cite{EtingofStoica}. Montarani deals with symplectic reflection algebras associated to wreath products, and the rational Cherednik algebra of $B_n$ is an example of such. Namely, Montarani proves that if $\ul{\lambda}$ extends to a representation of a rational Cherednik algebra at some parameters then $\ul{\lambda}$ is a rectangle concentrated in a single component, and she gives the formula for the parameters for which it extends depending on the height and width of the rectangle \cite{Montarani}, matching what we found in Theorem \ref{unitary+fd}. Etingof and Stoica prove that for the rational Cherednik algebra of a real reflection group, if $L(\ul{\lambda})$ is unitary then $L(\ul{\lambda})$ is in addition finite-dimensional if and only if $L(\ul{\lambda})=\ul{\lambda}$ \cite[Proposition 4.1]{EtingofStoica}. This implies Theorem \ref{unitary+fd}. Our proof is of independent interest as it uses different technology -- the combinatorial classification of unitary representations by Griffeth \cite{Griffeth} and the combinatorial classification of finite-dimensional representations arising from categorical affine Lie algebra actions \cite{ShanVasserot}. In particular, comparing the combinatorial classifications of unitary modules and finite-dimensional modules can in principle be done for $G(\ell,1,n)$ when $\ell>2$ (although the combinatorial classification in \cite{Griffeth} of unitary simples becomes very complicated for $\ell>2$). On the other hand, identifying unitary finite-dimensional simples by checking when $L(\ul{\lambda})=\ul{\lambda}$ is insufficient when $\ell>2$ as $G(\ell,1,n)$ is not a real reflection group and the criterion  \cite[Proposition 4.1]{EtingofStoica} need no longer hold for unitary finite-dimensional simple modules. So we may think of the case $\ell=2$ as a good test case providing evidence that all these different combinatorial gadgets are working compatibly and providing correct results, before we move on to the higher-level case.

\section{Background and notation}
A partition is a finite, non-increasing sequence of positive integers: $\lambda=(\lambda_1,\lambda_2,\dots,\lambda_r)$ such that $\lambda_1\geq \lambda_2\geq \dots\geq \lambda_r> 0$. Let $\mathcal{P}=\{\lambda\mid \lambda\hbox{ is a partition}\}\cup\{\emptyset\}$ be the set of all partitions including the empty partition $\emptyset$, which we consider as the unique partition of $0$. For a partition $\lambda$ we write $|\lambda|=\sum_{i=1}^r\lambda_i$ and we identify $\lambda$ with its Young diagram, the upper-left-justified array of $|\lambda|$ boxes in the plane with $\lambda_1$ boxes in the top row, $\lambda_2$ boxes in the second row from the top,..., $\lambda_r$ boxes in the $r$'th row from the top. A partition $\lambda$ is a rectangle if and only if $\lambda_1=\lambda_2=\dots=\lambda_r$. Denote the transpose partition of $\lambda$ by $\lambda^t$.

Let $B_n$ be the type $B$ Weyl group of rank $n$, also known as the hyperoctahedral group or the complex reflection group $G(2,1,n)$. Then $B_n\cong \left(\mathbb{Z}/2\mathbb{Z}\right)^n\rtimes S_n$ and the irreducible representations of $B_n$ over $\mathbb{C}$ are labeled by bipartitions of $n$ defined as $$\mathcal{P}^2(n):= \{\ul{\lambda}=(\lambda^1,\lambda^2)\mid \;|\lambda^1|+|\lambda^2|=n\}.$$
We call $\lambda^1$ and $\lambda^2$ the components of $\ul{\lambda}$. Set $$\mathcal{P}^2=\bigcup_{n\geq 0}\mathcal{P}^2(n)$$ and fix $e\in\N_{\geq 2}$ and $\ul{s}=(s_1,s_2)\in\mathbb{Z}^2$. The level $2$ Fock space $$\cF^2_{e,\ul{s}}=\bigoplus_{\ul{\lambda}\in\mathcal{P}^2}\mathbb{C}|\ul{\lambda},\ul{s}\rangle$$ of rank $e$ has basis given by ``charged bipartitions" $|\ul{\lambda},\ul{s}\rangle$. This means that we shift the contents of boxes in $\ul{\lambda}$ by the charge $\ul{s}$ so that the content of a box $b$ in row $x$, column $y$ of the Young diagram of $\lambda^j$ is $s_j+y-x=:\mathrm{ct}(b)$ for $j=1,2$. Let $s=s_2-s_1$. The Fock space is only defined up to a diagonal shift in $\ul{s}$, i.e. $\ul{s}$ and $\ul{s}+(a,a)$ yield the same Fock space for any $a\in\mathbb{Z}$, so unless otherwise noted we will always take $\ul{s}=(0,s)$.

The Fock space $\cF^2_{e,\ul{s}}$ comes endowed with the structure of an $\sle$-crystal as well as an exotic structure of an $\slinf$-crystal, and the combinatorial formulas for both actions depend on $\ul{s}$ and $e$. A crystal is a directed graph with at most one arrow between any two vertices. We call $|\ul{\lambda},\ul{s}\rangle$ a source vertex in a crystal if it has no incoming arrow. The vertices of both the $\sle$- and the $\slinf$-crystal are $\{|\ul{\lambda},\ul{s}\rangle\mid\ul{\lambda}\in\mathcal{P}^2\}$. 
These crystals come from a realization of $\cF^2_{e,\ul{s}}$ as the Grothendieck group of $$\cO_{e,\ul{s}}:=\bigoplus_{n\geq 0}\cO_{e,\ul{s}}(n)$$ where $\cO_{e,\ul{s}}(n)$ is the category $\cO$ of the rational Cherednik algebra of $B_n$ with parameters $(c,d)$ determined by the Fock space parameter $(e,\ul{s})$. The $\sle$-crystal on the Fock space arises via a categorical action of $\sle$ on $\cO_{e,\ul{s}}$; the Chevalley generators $f_i$ and $e_i$ act by $i$-induction and $i$-restriction functors, direct summands of the parabolic induction and restriction functors between $\cO_{e,\ul{s}}(n)$ and $\cO_{e,\ul{s}}(n+1)$ \cite{Shan}. Furthermore, there is a Heisenberg algebra action on $\cO_{e,\ul{s}}$ relevant for describing branching of irreducible representations with respect to 
categories $\cO$ attached to
 parabolic subgroups 
$B_m\times S_e^{\times k}$ \cite{ShanVasserot}. The categorical Heisenberg action gives rise to an $\slinf$-crystal structure on $\cF^2_{e,\ul{s}}$ whose arrows add $e$ boxes at a time to a bipartition.

\begin{theorem}\cite{ShanVasserot}\label{source vx} The irreducible representation $L(\ul{\lambda})\in \cO_{e,\ul{s}}(n)$ is finite-dimensional if and only if $\ul{\lambda}$ is a source vertex for both the $\sle$-crystal and the $\slinf$-crystal on $\cF^2_{e,\ul{s}}$.
\end{theorem}

A direct combinatorial rule for computing the arrows in the $\slinf$-crystal was given in \cite{GerberN} and builds off of previous partial results in this direction by Gerber \cite{Gerber1},\cite{Gerber2} and Losev \cite{Losev}. The rule uses abacus combinatorics. We define the abacus $\cA(\ul{\lambda})$ of a charged bipartition $|\ul{\lambda},\ul{s}\rangle$ as $\cA(\ul{\lambda})=\left(\{\beta^1_1,\beta^1_2,\beta^1_3,\dots\},\{\beta^2_1,\beta^2_2,\beta^2_3,\dots\}\right)$ where $\beta^j_i$ is the $i$'th $\beta$-number of $(\lambda^j)^t$ given by $\beta^j_i=(\lambda^j)^t_i+s_j-i+1$ if $i\leq r$ where $r$ is the number of parts of $(\lambda^j)^t$, i.e. the number of columns of $\lambda^j$, and $\beta^j_i=s_j-i+1$ if $i>r$. We remark that we have taken the transposes of $\lambda^1$ and $\lambda^2$ in order to be in agreement with the conventions of \cite{Griffeth} and \cite{EtingofStoica}, and this is the component-wise transpose of our convention in \cite{GerberN} and in the papers \cite{Gerber1},\cite{Gerber2}. It is convenient to visualize $\cA(\ul{\lambda})$ as an array of beads and spaces assembled in two horizontal rows and infinitely many columns indexed by $\mathbb{Z}$; the bottom row corresponds to $\lambda^1$ and the top row to $\lambda^2$, with a bead in row $j$ and column $\beta^j_i$ for each $i\in\mathbb{N}$ and each $j=1,2$, and spaces otherwise. Far to the left the abacus consists only of beads, far to the right only of spaces, so computations always occur in a finite region. The beads in $\cA(\ul{\lambda})$ which have some space to their left correspond to the nonzero columns of $\ul{\lambda}$.

The rational Cherednik algebra itself is a deformation of $\mathbb{C}[x_1,\dots,x_n,y_1,\dots,y_n]\rtimes \mathbb{C} B_n$ with multiplication depending on a pair of parameters $(c,d)\in \mathbb{C}^2$ \cite{EtingofGinzburg}; see \cite{Griffeth} for the definition used in \cite{Griffeth} to combinatorially classify unitary representations. The reader should be warned that there are many slight variants on the definition so that conventions are often slightly different from author to author, making precise numerical and combinatorial details of cyclotomic rational Cherednik algebras a minefield. We write $H_{e,\ul{s}}(B_n)$ for the rational Cherednik algebra depending on parameters $(c,d)$ as in \cite{Griffeth} where $c$ and $d$ are determined from the Fock space parameter $(e,\ul{s})$ by the formulas \cite{ShanVasserot}:
$$c=\frac{1}{e},\qquad d=-\frac{1}{2}+\frac{s}{e}.$$

The algebra $H_{e,\ul{s}}(B_n)$ has a category $\cO_{e,\ul{s}}(n)$ of representations which in particular contains all finite-dimensional representations \cite{GGOR}. The irreducible representations in $\cO_{e,\ul{s}}(n)$ are labeled by $\mathcal{P}^2(n)$ and denoted $L(\ul{\lambda})$ for $\ul{\lambda}\in\mathcal{P}^2(n)$ \cite{GGOR}. The irreducible representation $L(\ul{\lambda})$ comes with a non-degenerate contravariant Hermitian form and $L(\ul{\lambda})$ is called unitary if this form is positive definite \cite{EtingofStoica}. The unitary $L(\ul{\lambda})$ were classified combinatorially by Griffeth \cite[Corollary 8.4]{Griffeth}. In fact, he considers arbitrary parameters $(c,d)\in\mathbb{C}^2$ without any reference to the Fock space. It is known that the resulting category $\cO$ is equivalent to either the one arising from a level $2$ Fock space as above, or to a tensor product of level $1$ Fock spaces \cite{Rouquier}. We will only consider parameters coming from level $2$ Fock spaces as above, as the case of a product of level $1$ Fock spaces is well-understood. Such equivalences don't behave well on unitary bipartitions but they respect crystals, so the results of this paper probably imply the result for all parameters.

The element $\sum_{i=1}^n x_iy_i+y_ix_i\in H_{e,\ul{s}}(B_n)$ acts on the irreducible $\mathbb{C} B_n$-representation $\ul{\lambda}$, the lowest weight space of $L(\ul{\lambda})$, by a scalar $c_{\ul{\lambda}}$. We have the following formula for $c_{\ul{\lambda}}$:
$$c_{\ul{\lambda}}=|\lambda^1|+\frac{s}{e}\left(|\lambda^2|-|\lambda^1|\right)-\frac{2}{e}\sum_{b\in\ul{\lambda}}\mathrm{ct}(b)$$
where the sum runs over all boxes $b$ in $\ul{\lambda}$ and $\mathrm{ct}(b)$ includes the shift by $\ul{s}$ as explained above.

\begin{lemma}\cite{EtingofStoica}\label{lemma ES}
A unitary representation $L(\ul{\lambda})$ is finite-dimensional if and only if $c_{\ul{\lambda}}=0$.
\end{lemma}
\section{Rectangles parametrize finite-dimensional unitary irreducible representations}\label{section rectangle}
This section consists of the proof of Theorem \ref{unitary+fd}. The proof is broken into two steps. First, we use Lemma \ref{lemma ES} together with the conditions in \cite[Corollary 8.4]{Griffeth} to characterize the unitary finite-dimensional irreducible representations labeled by bipartitions with only one non-empty component. Second, we show that if $\lambda^1$ and $\lambda^2$ are both non-empty partitions, then $L(\lambda^1,\lambda^2)$ is never both unitary and finite-dimensional. 

\begin{proposition}\label{rectangle}
Fix $\ul{\lambda}\in\mathcal{P}^2$ such that $\ul{\lambda}=(\lambda,\emptyset)$. Then $L(\ul{\lambda})$ is a finite-dimensional and unitary $H_{e,\ul{s}}(B_n)$-representation for some Fock space parameter $(e,\ul{s})$ if and only if $\lambda$ is a rectangle. Conversely, given any rectangle $\lambda$, for each $e\geq 2$ there is a unique charge $\ul{s}$ (up to shift) such that $L(\lambda,\emptyset)$ is finite-dimensional and unitary.
\end{proposition}

\begin{proof}
Assume that $c_{\ul{\lambda}}=0$. Then we have %$n(e-s)=2\sum_b\mathrm{ct}(b)$ or, 
$$\frac{s}{e}=\frac{1}{n}\left( n-\frac{2}{e}\sum_{b\in\ul{\lambda}}\mathrm{ct}(b)\right)=1-\frac{2}{en}\sum_{b\in\ul{\lambda}}\mathrm{ct}(b)$$
and thus the parameters for the rational Cherednik algebra are 
$$c=\frac{1}{e},\qquad d=\frac{1}{2}-\frac{2}{en}\sum_{b\in\ul{\lambda}}\mathrm{ct}(b)$$
Since $s$ is a function of $e$ and $n$, it is clear that there if there exists $\ul{s}$ such that $L(\lambda)$ is unitary, then $\ul{s}$ is unique up to adding $(a,a)$.% (which does not change the parameter for the Cherednik algebra).

We suppose that $c_{\ul{\lambda}}=0$ and check when conditions (a)-(e) of \cite[Corollary 8.4]{Griffeth} can hold in order to see when $L(\ul{\lambda})$ is unitary; by Lemma \ref{lemma ES} this is equivalent to checking when $L(\ul{\lambda})$ is both finite-dimensional and unitary. 
For an integer $m$, consider the quantity
$$d+mc=\frac{1}{2}+\frac{1}{e}\left(m-\frac{2}{n}\sum_b\mathrm{ct}(b)\right)$$
under our assumption $c_{\ul{\lambda}}=0$. The inequality $d+mc\leq \frac{1}{2}$ holds if and only if $m\leq \frac{2}{n}\sum_b\mathrm{ct}(b)$, and $d+mc=\frac{1}{2}$ if and only if $m=\frac{2}{n}\sum_b\mathrm{ct}(b)$. 

Let $b_1$ be the box of largest content in $\lambda$.
If $\lambda$ has $q$ columns then $\mathrm{ct}(b_1)=q-1$. The inequality $d+\mathrm{ct}(b_1)c\leq \frac{1}{2}$ holds if and only if $q-1\leq \frac{2}{n}\sum_b\mathrm{ct}(b)$. If $q=n$ then $\lambda=(n)$ and we have equality, and moreover, $b_2=b_1$ so case (d) holds and $L(n)$ is unitary. 
Suppose next that $\lambda$ is a rectangle with $q$ columns and $r$ rows where $r>1$. So $n=qr$ and $\frac{2}{n}\sum_b\mathrm{ct}(b)=\frac{2}{n}\left(\frac{n(q-r)}{2}\right)=q-r$. If $r>1$ then we have $q-1>\frac{2}{n}\sum_b\mathrm{ct}(b)$, and thus $d+\mathrm{ct}(b_1)c>\frac{1}{2}$. Finally, if $\lambda$ is not a rectangle but its first row has $q$ boxes, then clearly $q-r>\frac{2}{n}\sum_b\mathrm{ct}(b)$ so we also have $q-1>\frac{2}{n}\sum_b\mathrm{ct}(b)$. It follows that for arbitrary $\lambda\neq (n)$, cases (a) and (b) of \cite[Corollary 8.4]{Griffeth}  cannot occur. 

The remaining three cases (c),(d),(e) of \cite[Corollary 8.4]{Griffeth} all require the equation $d+mc=\frac{1}{2}$ to be satisfied for some integer $m$, so equivalently, $m=\frac{2}{n}\sum_{b}\mathrm{ct}(b)$ for some integer $m$. If $\lambda$ is a rectangle with $q$ columns and $r$ rows, $r>1$, then the solution is $m=q-r=\mathrm{ct}(b_2)$, so case (d) holds and $L(\lambda,\emptyset)$ is unitary. Otherwise, observe that cases (c),(d),(e) all require $m\geq \mathrm{ct}(b_2)$, where $b_2$ is the removable box of largest content. Writing $\lambda=(q^r, \lambda_{r+1},\lambda_{r+2}\dots)$, $q>\lambda_{r+1}$, we have $\mathrm{ct}(b_2)=q-r$ which is $2$ times the average content of the boxes contained in the $q$ by $r$ rectangle comprising the first $r$ rows of $\lambda$. Adding boxes below this rectangle clearly lowers the average content of the boxes in the diagram, so we always have $m<\mathrm{ct}(b_2)$ if $\lambda$ is not a rectangle. Therefore cases (c),(d),(e) cannot hold if $\lambda$ is not a rectangle.

We conclude that given $e\geq 2$ and a rectangle $\lambda$ with $r$ rows and $q$ columns, $L(\lambda,\emptyset)$ is finite-dimensional and unitary exactly when $s=e-q+r$; and if $\lambda$ is not a rectangle, then $L(\lambda,\emptyset)$ is never finite-dimensional and unitary.
\end{proof}

\begin{remark}
If $\ul{\lambda}=(\lambda,\emptyset)$ where $\lambda$ is a rectangle with $q$ columns and $r$ rows then $s=e-q+r$ is the smallest value of $s$ for which $L(\ul{\lambda})$ is finite-dimensional \cite{GerberN}.
\end{remark}
\begin{remark}\label{switch} Switching the components $\lambda^1$ and $\lambda^2$ is induced by an isomorphism of the underlying rational Cherednik algebras sending $d$ to $-d$; on the charge for the Fock space it sends $s$ to $e-s$. Thus an analogous result to Proposition \ref{rectangle} holds for $(\emptyset, \lambda)$ and we leave the formula to the reader.
\end{remark}

Next, we consider the case that both components of $\ul{\lambda}$ are non-empty. 
We will make an abacus argument. Following \cite[Definition 2.2]{JaconLecouvey}, given $\cA(\ul{\lambda})$ the abacus of $|\ul{\lambda},\ul{s}\rangle$, we say that $\cA(\ul{\lambda})$ has an $e$-period if there exist $\beta^{j_1}_{i_1},\beta^{j_2}_{i_2},\dots,\beta^{j_e}_{i_e}\in\cA(\ul{\lambda})$ such that $j_1\geq j_2\geq\dots\geq j_e$, $\beta^{j_{m+1}}_{i_{m+1}}=\beta^{j_m}_{i_m}-1$ for each $m=1,\dots,e-1$, $\beta^{j_1}_{i_1}$ is in the rightmost column of $\cA(\lambda)$ containing a bead, and for each $m=1,\dots,e$ if $j_m=2$ then $\beta^2_{i_m}$ has an empty space below it. Let $\mathrm{Per}^1=\{\beta^{j_1}_{i_1},\beta^{j_2}_{i_2},\dots,\beta^{j_e}_{i_e}\}\subset\cA(\ul{\lambda})$ denote the $e$-period of $\cA(\ul{\lambda})$ if it exists, and call it the first $e$-period. The $k$'th $e$-period $\mathrm{Per}^k$ of $\cA(\ul{\lambda})$ is defined recursively as the $e$-period of $\left(\dots\left(\cA({\lambda})\setminus \mathrm{Per^1}\right)\setminus\mathrm{Per^2}\dots\right)\setminus\mathrm{Per}^{k-1}$ if it exists and $\mathrm{Per}^1,\dots,\mathrm{Per}^{k-1}$ exist. We
say that $\cA(\ul{\lambda})$ is totally $e$-periodic if $\mathrm{Per}^k$ exists for any $k\in\N$ (see \cite[Definition 5.4]{JaconLecouvey}). The charged bipartition $|\ul{\lambda},\ul{s}\rangle$ is a source vertex for the $\sle$-crystal on $\mathcal{F}^2_{e,\ul{s}}$ if and only if $\cA(\ul{\lambda})$ is totally $e$-periodic \cite[Theorem 5.9]{JaconLecouvey}. 
In \cite{GerberN} we gave a criterion for checking if $|\ul{\lambda},\ul{s}\rangle$ is a source vertex in the $\slinf$-crystal in terms of a pattern avoidance condition on $\cA(\ul{\lambda})$ \cite[Theorem 7.13]{GerberN}. 

\begin{lemma}\label{>=e col} If $\ul{\lambda}=(\lambda^1,\lambda^2)$ with both $\lambda^1\neq \emptyset$ and $\lambda^2\neq \emptyset$ and the abacus $\cA(\ul{\lambda})$ of $|\ul{\lambda},\ul{s}\rangle$ is totally $e$-periodic, then $\ul{\lambda}$ has at least $e$ nonzero columns. 
\end{lemma}
\begin{proof}Suppose $\cA(\ul{\lambda})$ is totally $e$-periodic and consider the maximal beta-number $\beta^j_1$ in $\cA(\ul{\lambda})$ (possibly it occurs twice, as $\beta^2_1$ and $\beta^1_1$). It must correspond to at least one nonzero column of $\ul{\lambda}$ since both components of $\ul{\lambda}$ are nonempty. If $\beta^1_1$ is maximal then the whole first $e$-period $\mathrm{Per}^1$ lies in the bottom row of $\cA(\ul{\lambda})$. Since $\lambda^1\neq \emptyset$, $\mathrm{Per}^1$ corresponds to $e$ nonzero columns of $\lambda^1$ of the same size. If $\beta^2_1$ is maximal then either all of $\mathrm{Per}^1$ lies in the top row of $\cA(\ul{\lambda})$ and corresponds to $e$ nonzero columns of $\lambda^2$ of the same size, or the first $a$ beads of $\mathrm{Per}^1$ are $\beta^2_1,\dots \beta^2_a$ lying in the top row   and these must correspond to $a$ nonzero columns of $\lambda^2$ since $\lambda^2\neq \emptyset$; while the remaining $e-a$ beads of $\mathrm{Per}^1$ are $\beta^1_{1},\dots, \beta^1_{e-a}$ lying in the bottom row and these must correspond to $e-a$ nonzero columns of $\lambda^1$ since $\lambda^1\neq \emptyset$. So $\ul{\lambda}$ has at least $e$ columns.
\end{proof}

\begin{lemma}\label{e col} If $\ul{\lambda}=(\lambda^1,\lambda^2)$ with both $\lambda^1\neq \emptyset$ and $\lambda^2\neq \emptyset$ and $\ul{\lambda}$ has exactly $e$ nonzero columns then $L(\ul{\lambda})$ is not finite-dimensional.
\end{lemma}
\begin{proof}Suppose $L(\ul{\lambda})$ is finite-dimensional, then $\cA(\ul{\lambda})$ is totally $e$-periodic by \ref{source vx} and \cite[Theorem 5.9]{JaconLecouvey}.
%Suppose $\ul{\lambda}$ has exactly $e$ nonzero columns.
 Then the beads of $\cA(\ul{\lambda})$ labeling the nonzero columns of $\ul{\lambda}$ must comprise the first $e$-period $\mathrm{Per}^1$ of $\cA(\ul{\lambda})$, and $\mathrm{Per}^1$ is as at the end of the previous proof where the first $a$ beads of $\mathrm{Per}^1$ are in the top row of $\cA(\ul{\lambda})$ (corresponding to $a$ nonzero columns of $\lambda^2$), and the last $e-a$ beads of $\mathrm{Per}^1$ are in the bottom row of $\cA(\ul{\lambda})$.  
But then the space in the bottom row to the left of $\beta_{e-a}^1$ and the space in the top row to the left of $\beta_a^2$ form a pair of spaces violating the condition in \cite[Theorem 7.13]{GerberN}, so $|\ul{\lambda},\ul{s}\rangle$ is not a source vertex in the $\slinf$-crystal, and in particular, $L(\ul{\lambda})$ is not finite-dimensional by Theorem \ref{source vx}. 
\end{proof}

\begin{proposition}\label{nonempty components} If $\ul{\lambda}=(\lambda^1,\lambda^2)$ with both $\lambda^1\neq \emptyset$ and $\lambda^2\neq \emptyset$ then $L(\ul{\lambda})$ is never both unitary and finite-dimensional.
\end{proposition}

\begin{proof}
We consider the cases in \cite[Corollary 8.5]{Griffeth} one by one. 

Case (a). Converting to the Fock space parameters, the inequalities are:
\begin{align*}&-s\leq \#\{\hbox{columns}(\lambda^1)\}+\#\{\hbox{rows}(\lambda^2)\}-1\leq e-s,\\
&s-e\leq \#\{\hbox{columns}(\lambda^2)\}+\#\{\hbox{rows}(\lambda^1)\}-1\leq s.
\end{align*}
Since both $\#\{\hbox{rows}(\lambda^1)\}\geq 1$ and $\#\{\hbox{rows}(\lambda^2)\}\geq 1$ by the assumption $\lambda^1, \lambda^2\neq \emptyset$, when we add the inequalities we see that $\#\{\hbox{columns}(\ul{\lambda})\}\leq e$. %\hbox{ and }0<s<e.$$ 
Suppose $L(\ul{\lambda})$ is finite-dimensional. Then $\cA(\ul{\lambda})$ is totally $e$-periodic by Theorem \ref{source vx} and \cite[Theorem 5.9]{JaconLecouvey}. By Lemma \ref{>=e col} then $\ul{\lambda}$ has at least $e$ columns. So $\ul{\lambda}$ has exactly $e$ columns, but then by Lemma \ref{e col}, $L(\ul{\lambda})$ cannot be finite-dimensional.

Case (b). If $d+\ell c=\frac{1}{2}$ then $\ell=e-s$. It is required that $\ct(b_2)-\ct(b_4')+1\leq \ell\leq \ct(b_1)-\ct(b_4')+1$, which translates to (ignoring the rightmost term),
$$ \#\{\hbox{col}(\lambda^1)\}-\#\{\hbox{row}(\lambda^1)\hbox{ of size }\lambda_1^1\} + \#\{\hbox{row}(\lambda^2)\}\leq e-s.$$
Additionally, the inequality $-d+(\ct(b_1')-\ct(b_4)+1)c\leq\frac{1}{2}$ must hold, which implies
$$ \#\{\hbox{col}(\lambda^2)\}+\#\{\hbox{row}(\lambda^1)\}-1\leq s.$$
Adding inequalities, we have $$\#\{\hbox{col}(\ul{\lambda})\}\leq \{\hbox{col}(\ul{\lambda})\}+\left(\#\{\hbox{row}(\lambda^1)\}-\#\{\hbox{row}(\lambda^1)\hbox{ of size }\lambda_1^1\}\right) + \left(\#\{\hbox{row}(\lambda^2)-1\right)\leq e.$$ 
Now we conclude as in case (a) that $L(\ul{\lambda})$ is not finite-dimensional.

Case (c). If $-d+\ell c=\frac{1}{2}$ then $\ell=s$. The first inequality is $\ct(b_2')-\ct(b_4)+1\leq \ell\leq \ct(b_1')-\ct(b_4)+1$ which yields (ignoring the rightmost term) the inequality
$$\#\{\hbox{col}(\lambda^2)\}-\#\{\hbox{row}(\lambda^2)\hbox{ of size }\lambda_1^2\} + \#\{\hbox{row}(\lambda^1)\}\leq s .$$
Next, the inequality $d+(\ct(b_1)-\ct(b_4')+1)c\leq \frac{1}{2}$ yields $$\#\{\hbox{col}(\lambda^1)\}+\#\{\hbox{row}(\lambda^2)\}-1\leq e-s.$$ 
Now we add the inequalities and conclude as in case (b) that $L(\ul{\lambda})$ is not finite-dimensional.

Cases (d) and (e). These have $c$ replaced with $-c$ and $\ct(b)$ replaced with $-\ct (b)$ in all conditions, the latter being the same as considering $\lambda^t$ with the original conditions. But to deal with a Fock space with $-c$ instead of $c$ we just take charge $-\ul{s}$ and replace $\ul{\lambda}$ with $\ul{\lambda}^t$ \cite[Section 4.1.4]{Losev}. So cases (d) and (e) reduce to cases (b) and (c). 

Case (f). As in case (b), $d+\ell c=\frac{1}{2}$ implies $\ell=e-s$; and $ -d+mc=\frac{1}{2}$ implies $m=s$. We then have the inequalities:
$$\#\{\hbox{col}(\lambda^1)\}-\#\{\hbox{row}(\lambda^1)\hbox{ of size }\lambda_1^1\} + \#\{\hbox{row}(\lambda^2)\}\leq e-s,$$
$$\#\{\hbox{col}(\lambda^2)\}-\#\{\hbox{row}(\lambda^2)\hbox{ of size }\lambda_1^2\} + \#\{\hbox{row}(\lambda^1)\}\leq s.$$
We add the inequalities and conclude as in case (b).

Case(g). This case actually does not arise in the Fock space set-up: first, our assumption $c=\frac{1}{e}$, $d=-\frac{1}{2}+\frac{s}{e}$ determines $\ell=e-s$ and $m=s$ as in case (f). Then observe that $b_3$ is just ``$b_2$" for $(\lambda^1)^t$, etc, and the inequalities of case (g) if multiplied through by $-1$ become the inequalities of case (f), but for the transpose bipartition $\ul{\lambda}^t=((\lambda^1)^t, (\lambda^2)^t)$. But then we would have $\#\{\hbox{col}(\ul{\lambda}^t)\}\leq -e$, which is nonsense since $e>0$.
\end{proof}

Combining Propositions \ref{rectangle} and \ref{nonempty components}, we conclude that Theorem \ref{unitary+fd} holds.

%  It is interesting to compare the classification given by Theorem \ref{unitary+fd} with a result of Etingof and Montarani: these irreducible representations labeled by rectangles concentrated in a single component arise as deformations of finite-dimensional representations of a wreath product algebra \cite{EtingofMontarani}.

%\begin{question} 
%Consider $G(\ell,1,n)$ instead of $G(2,1,n)=B_n$. Both the classification of unitary representations and the classification of finite-dimensional representations is more complicated and cannot be summarized in a closed formula. Nevertheless it seems natural to ask: are the finite-dimensional unitary representations of $H_c(G(\ell,1,n)$ at parameters $c$ corresponding to parameters $e,(s_1,\dots,s_\ell)$ for the level $\ell$  Fock space labeled by $\ell$-partitions of the form $(\emptyset,\dots,\lambda,\dots,\emptyset)$ where $\lambda$ is a rectangle?
%\end{question}

\textbf{Acknowledgments.} Thanks to S. Griffeth for asking the question about the intersection of the unitary locus of parameters with the finite-dimensional locus, and thanks to C. Bowman, T. Gerber, S. Griffeth, and J. Simental for useful discussions. I would like to especially thank J. Simental for bringing to my attention S. Montarani's work in \cite{Montarani}.

\end{document}